\documentclass[10pt]{article}
\usepackage{amsthm,amsmath,amssymb}
\usepackage{comment}
\usepackage{amsfonts,amsmath,amsthm,amssymb,comment} 
\usepackage{authblk}
\usepackage{mathrsfs}

\newtheorem{thm}{Theorem}[section]
\newtheorem{prop}[thm]{Proposition}
\newtheorem{lemma}[thm]{Lemma}

\newtheorem{guess}{Conjecture}

\newtheorem{example}{Example}[section]
\newtheorem{problem}{Problem}

\def\qed{\hfill $\Box$ \vskip .15 in}
\def\ZZ{{\mathbb Z}}

\newcommand{\zed}{\ensuremath{\mathbb{Z}}}
\newcommand{\A}{\mathscr{A}}

\begin{document}

\title{\bf On Partial Sums in Cyclic Groups} 

\author[1]{Dan S.\ Archdeacon}
\author[1]{Jeffrey H.\ Dinitz}
\author[1]{Amelia Mattern}
\author[2]{Douglas~R.~Stinson\thanks{D.~Stinson's research is supported by NSERC discovery grant 203114-11.}}
\affil[1]{Department of Mathematics and Statistics\\
University of Vermont\\
Burlington, VT 05405
U.S.A.}
\affil[2]{David R.\ Cheriton School of Computer Science\\ University of Waterloo,
Waterloo, Ontario, N2L 3G1, Canada}

\date{\today}        
\maketitle

\begin{abstract}
We are interested in ordering the elements of a subset $A$ of the non-zero integers
modulo $n$ in such a way that all the partial sums are distinct. We conjecture that this
can always be done and we prove various partial results about this problem.
\end{abstract}

\section{Introduction}

%Give motivation from topological perspective from 1st paper on Heffter arrays.

Suppose that $A = \{ a_1, \dots , a_k\} 
\subseteq \ZZ_n \setminus \{0\}$ is a subset of the integers modulo $n$. Let $(a_1, a_2,
\ldots, a_k)$ be an ordering of the elements in $A$. Define the {\em partial sums} $s_1, \dots , s_k$ 
by the formula $s_j = \sum_{i=1}^j  a_i$ ($1 \leq j \leq k$), where all arithmetic is in $\ZZ_n$.
We propose the following conjecture.

\begin{guess}\label{conjecture1}
Suppose $A = \{ a_1, \dots , a_k\}  \subseteq \ZZ_n \setminus \{0\}$. Then
  there exists an ordering of the elements of $A$ 
  such that the partial sums are all distinct, i.e., 
 $1\leq i <j\leq k$ implies $s_i \neq s_j$.  
  \end{guess}
  
  \begin{example}
 Suppose we have $A = \{1,2,3,4,5,6\} \subseteq \zed_8$.
 Consider the ordering:
 \[ 1 \: \:  \: 6 \:  \: \:  3 \: \:  \:  4 \:  \: \:  5 \:  \: \:  2.\]
 Then the partial sums are
 $ 1 \:  \: \:  7 \:  \: \:  2 \:  \: \:  6 \:  \: \:  3 \:  \: \:  4$, which are all distinct.
 \end{example}

We call an ordering of a set $A$ a {\em sequencing} if  all of the partial sums are distinct. Our interest in Conjecture \ref{conjecture1} 
 was motivated by a recent construction due to Archdeacon (see \cite{A14})
 for embedding complete graphs
so the faces are 2-colorable and each color class is a $k$-cycle system.  If Conjecture \ref{conjecture1} is true, then given any $k$-subset $A \subseteq \ZZ_n \setminus \{0\}$ there exists a cyclic $k$-cycle system on the Cayley graph consisting of the edges in $K_n$ whose lengths are in $A$.

Alspach was interested in a similar decomposition problem, but with paths of length $k$ instead of $k$-cycles.   The following slightly different conjecture was made several years ago
by Alspach  (see \cite{BH05}):
 
\begin{guess}\label{conjecture2} (Alspach)
Suppose $A = \{ a_1, \dots , a_k\}  \subseteq\ZZ_n \setminus \{0\}$ has the property that
$\sum _{a \in A}a \neq 0$. Then
  there exists an ordering of the elements of $A$ 
  such that the partial sums are all distinct and nonzero.  
  \end{guess}

In the following proposition, we show that Conjecture \ref{conjecture2} implies Conjecture \ref{conjecture1}

\begin{prop}Conjecture 2 implies Conjecture 1.
\end{prop}

\proof Assume that Conjecture \ref{conjecture2} is true.  Let $A = \{ a_1, \dots , a_k\}  \subseteq\ZZ_n \setminus \{0\}$.  If $\sum _{a \in A}a \neq 0$, then by Conjecture 2 there is an ordering of the elements of $A$ 
  such that the partial sums are all distinct, proving Conjecture 1 in this case.

So assume that $\sum _{a \in A}a =0$.  It follows that $  \sum_{i=1}^{k-1}  a_i \neq 0$. So by Conjecture 2 there is an ordering  $(a_1', a_2', \ldots, a_{k-1}')$ of $\{ a_1, \dots , a_{k-1}\}$ where all of the partial sums are distinct and nonzero.  Now reinsert $a_k$ at the end of the ordering to get $(a_1', a_2', \ldots, a_{k-1}', a_k)$.  The only new  partial sum is
$s_k=0= \sum _{a \in A}a $ and since all of the earlier partial sums are nonzero (and distinct), we have that  all the partial sums are now distinct. This proves Conjecture 1. \qed

Both conjectures can be 
 be restated in terms of runs. As before, we let $(a_1, a_2, \ldots, a_k)$ be an ordering
 of $A = \{ a_1, \dots , a_k\} \subseteq \zed_n \setminus \{0\}$.
 Let $i,j$ be integers such that $1 \leq i < j \leq k$.
 We define the {\it run} $r_{i,j}$ by the formula
 \[r_{i,j} = \sum _{h = i}^j a_h.\]
 For $1 \leq i < j \leq k$, it is obvious that
 
 \begin{equation}
 \label{run-sum}
 s_i = s_j \Leftrightarrow r_{i+1,j} = 0.
 \end{equation}
 
 The following conjecture is easily seen to be equivalent to Conjecture \ref{conjecture1}.
 
 \begin{guess}\label{conjecture3}
Suppose $A = \{ a_1, \dots , a_k\}  \subseteq\ZZ_n \setminus \{0\}$. Then
  there exists an ordering of the elements of $A$ 
  such that the runs $r_{i,j}$ are nonzero,  for all $i,j$ with
  $2 \leq i < j < n$.  
  \end{guess}
  
  Note that we allow a run $r_{1,j}$ (which is just a partial sum) to equal 0 in  Conjecture \ref{conjecture3}.
Similarly, Conjecture \ref{conjecture4} is easily seen to be equivalent to Conjecture \ref{conjecture2}.
 
 \begin{guess}\label{conjecture4}
Suppose $A = \{ a_1, \dots , a_k\}  \subseteq\ZZ_n \setminus \{0\}$ 
has the property that
$\sum _{a \in A} \neq 0$. Then
  there exists an ordering of the elements of $A$ 
  such that the runs $r_{i,j}$ are nonzero,  for all $i,j$ with
  $1 \leq i < j < n$.  
  \end{guess}
 
 %Our interest in Conjecture \ref{conjecture1} 
 %was motivated by a recent construction due to Archdeacon (see \cite{A14})
 %for embedding complete graphs
%so the faces are 2-colorable and each color class is a cycle system.

Conjectures \ref{conjecture1} and \ref{conjecture2} are also natural generalizations of 
sequenceable and $R$-sequenceable groups. 
 A group $G$ is \emph{sequenceable} if there exists an ordering of \emph{all} the group
elements such that all the partial sums are distinct.
It is well-known that $(\zed_n,+)$ is sequenceable if and only if $n$ is even.
More generally, it is known 
that an abelian group is sequenceable if and only if it has a unique element of order 2.   
When $n$ is odd, $(\zed_n,+)$ cannot be sequenced because the sum of all the group elements 
is zero (the first element in the sequencing must be 0, so the first and last sums both equal zero).
However, it has been shown that $(\zed_n,+)$ is {\it $R$-sequenceable} when $n$ is odd (this 
allows the first and last sums to both equal zero).
For references to proofs of these results, see the survey by Ollis \cite{O13}.

Conjecture \ref{conjecture1}  can be considered as a sequencing of an arbitrary subset of 
the non-zero elements of the cyclic group $(\zed_n,+)$. Since there are
$2^{n-1} - 1$ nonempty subsets of   $\zed_n \setminus \{0\}$, there are 
many problems to be solved for each $n$. The lack of structure (in general)
of these subsets is perhaps what makes the problem seemingly difficult to solve.

The only published work on this problem is by Bode and Harborth \cite{BH05}.
They state without proof that Conjecture \ref{conjecture2} 
is valid if $k \leq 5$ or
if $n \leq 16$ (the latter was obtained by computer verification). They also prove that 
Conjecture \ref{conjecture2}  is true if $k =n-1$ or $ n-2$.

We would also like to note an interesting and related unpublished conjecture by Marco Buratti.

\begin{guess} (Buratti) \label{buratti-conj}
Given $p$ a prime and a multiset $S$ containing $p-1$
non-zero elements from $\ZZ_p$, there exists an arrangement of $S$ so that all of the partial sums are distinct in $\ZZ_p$.
\end{guess}

Horak and Rosa \cite{HR09} generalized Buratti's conjecture to general cyclic groups (an additional condition was added). Some followup work was done by Pasotti and Pelligrini \cite{PP14.1,PP14.2}.
Almost simultaneously, Dinitz and Janiszewski \cite{DJ09} examined a special case of Buratti's Conjecture. 

In the remainder of this paper, we describe the results we have obtained on Conjecture \ref{conjecture1} and some related problems.
These results can be summarized as follows:

\begin{itemize}
\item A computer verification of Conjecture \ref{conjecture1} for $n \leq 25$ (see Section \ref{sec2}).
\item A proof of Conjecture \ref{conjecture1} for $k \leq 6$ (see Section \ref{sec3}).
\item Some results on ordering random subsets of a given size $k$ (see Section \ref{sec4}).
\item Some results on ordering {\it subsets} of a $k$-subset $A$ (see Section \ref{sec5}).
\item Some results on the number of $k$-subsets having a given sum, when $n$ is prime (see Section \ref{sec6}).
\end{itemize}

\section{Computer Verifications for Small  $n$}
\label{sec2}

We have checked that Conjecture \ref{conjecture1} is true up to $n=25$.   The algorithm is easy to describe.  For each subset  $A\subseteq \ZZ_n \setminus \{0\}$, we choose a random permutation of the elements of $A$.  If that ordering does not yield distinct partial sums, then we choose another random permutation.  We repeat this process until we  find an ordering of $A$ that gives distinct partial sums.  When $|A|$ is small with respect to $n$, we generally only need to choose a very few random permutations before a solution is found. However, when $|A|$ is close to $n$, many random permutations usually must be tried.  For example, when $n=25$, we needed fewer than 6 tries for nearly all subsets $A$ with $|A| \leq 7$. We used fewer than 100 tries when $|A| \leq 13$ and fewer than 10,000 tries when $|A| \leq 18$. However, when $|A| \geq 22$, there are cases where over 300,000 permutations were tried before a successful ordering was found. In general between 10,000 and 75,000 permutations were checked before finding a solution in these cases.  This algorithm was programmed in Mathematica and run on a laptop PC.  It found all the orderings of the subsets of $\ZZ_{24} \setminus \{0\}$ in roughly 3 days.  The orderings of the subsets of $\ZZ_{25} \setminus \{0\}$ took longer.
  
% did 24 on 3/2/14 -- took about three days. 
% did 25 while on sabbatical -- had it running on my office maching in VT.  Finished when I returned -- no idea of total time.

\section{Proof of Conjecture \ref{conjecture1} for $k \leq 6$}
\label{sec3}

Next we show that Conjecture \ref{conjecture1} is always true if the subset $A$ is small, independently supporting results in \cite{BH05}. 

\begin{thm} Conjecture \ref{conjecture1} is true when $k \leq 5$. \end{thm}

\begin{proof} This is easy to show for $k =1,2,3$.  

Assume $k =4$. Let $A = \{a_1, a_2, a_3, a_4\}$. Let $p$ be the number of pairs $\{x,-x\}$ in $A$.  
So $p = 0, 1$ or 2.  To sequence the set first choose three elements (renaming if necessary) so that $s_1,s_2$ and $s_3$ are distinct.

Assume $p=0$.  
Clearly $s_4 \neq s_3$ and since $p = 0$ we get that $s_4 \neq s_2$.  If $s_4 \neq s_1 (=a_1)$ we are done, so assume $s_4 = s_1$.  So $a_2+a_3+a_4 = 0$.  Now consider the ordering  $(a_1', a_2', a_3',a_4')$, where $a_1'=a_2,\ a_2'=a_1,\ a_3'=a_3$, and $a_4'=a_4$. Let $s_j'$ be the sum of the first $j$ terms in this new sequence.  Note we only need to check that $s_1'\neq s_4'$, since $p=0$.  Now this fails only if $a_1+a_3+a_4 = 0$, but from above we have that $a_2+a_3+a_4 = 0$, hence it only fails if $a_1 =a_2$ which is not the case.

Assume $p=1$.  Let $A = \{x,-x,y,z\}$.  Then the ordering $(z,x,y,-x)$ has partial sums $z, z+x, z+x+y, z+y$ and these are all distinct. 

Assume $p=2$. Let $A = \{x,-x,y,-y\}$.  Here the ordering $(x,y,-x,-y)$ works.

\bigskip

Now assume that $k=5$ with $A = \{a_1, a_2, a_3, a_4, a_5\}$. Again let $p$ be the number of pairs $\{x,-x\}$ in $A$ and so as before $p = 0, 1$ or 2. Again order A so that $s_1, s_2$ and $s_3$ are distinct.

Assume $p=0$ and that $A$ has been ordered in the natural way. In this case since there are no occurrences of a pair $\{x, -x\}$, we see that $s_i \neq s_{i+2}$ for $i = 1,2,3$. So the only conditions that can  fail are the following three possibilities: (1) $s_1 =s_4$, (2) $s_2 =s_5$ or  (3) $s_1 =s_5$.  It is straightforward to show that if any one of these conditions hold, then the other two do not hold.  We look at each case individually.

\begin{enumerate}
  \item { ($s_1 =s_4$):}  In this case we get that $a_2+a_3+a_4 = 0$.  Now order $A$ as $(a_1', a_2', a_3',a_4',a_5')=(a_1, a_2, a_3, a_5, a_4)$.  Then checking the conditions we see that $s_1' \neq s_4'$ since this would imply that $a_2+a_3+a_5=0$, however since $a_2+a_3+a_4 = 0$ this can not happen.  Next we note that  $s_1' \neq s_5'$ since $s_1' =s_1 \neq s_5 =s_5'$.  Similarly $s_2' \neq s_5'$.  
  
\item { ($s_2 =s_5$):} Here we order $A$ as $(a_1, a_3, a_2,a_4,a_5)$.  The verifications are similar to the previous case.

\item { ($s_1 =s_5$):}  In this case we get that $a_2+a_3+a_4 +a_5= 0$.  Order $A$ as $(a_1', a_2', a_3',a_4',a_5')=(a_2, a_1, a_3, a_4, a_5)$.  Again checking the conditions we see that $s_1' \neq s_5'$ as $a_1+a_3+a_4+a_5 \neq 0 $ since $a_2+a_3+a_4+a_5 = 0$.  Next we get that $s_2' \neq s_5'$ since $s_2' =s_2 \neq s_5 = s_5'$.  Finally if $s_1' =s_4'$, this implies that $a_1+a_3+a_4 =0$ which may happen.  If this is the case we reorder as follows: $(a_1'', a_2'', a_3'',a_4'',a_5'')=(a_3, a_2, a_1, a_4, a_5)$.  Looking at the three cases we see that since $a_1+a_3+a_4 =0$, then $s_1'' \neq s_4''$.  Next, $s_1'' \neq s_5''$ since $s_1 = s_5$. Finally $s_2'' \neq s_5''$, since $a_1+a_3+a_4 =0$.  This completes the case of $p=0$.

\end{enumerate}

Assume  that $p=1$. Let $A = \{x,-x,y,z,w\}$, ordered as $(z,x,y,-x,w)$.  The values of $s_i$ for $1\leq i\leq 5$, are
$z, z+x, z+x+y,z+y$, and $z+y+w$, respectively.  The only way that two of these values can be equal would be if $x = y+w$.  Assume this to be the case.  Now reorder $A$ as $(z,-x,y,x,w)$.  Here the values of $s_i$ for $1\leq i\leq 5$, are
$z, z-x, z-x+y,z+y$, and $z+y+w$, respectively. The only way for two of these values to be equal would be if $-x = y+w$.  But since 
$x = y+w$ this can not happen.  Thus a suitable ordering is always possible in this case.

Finally, assume that $p=2$.   Let $A = \{x,-x,y,-y,z\}$ and now order $A$ in the following two ways: $(x,y,z,-y,-x)$ and $(-x,y,z,-y,x)$. The first way fails only if $x=z-y$, while the second fails only if $-x=z-y$.  Since $x$ and $-x$ are distinct in this case, we have that one of these two orderings will always work.  This completes the proof.
\end{proof}

In the next theorem we prove Conjecture 1 when $k =6$.  
%Clearly there are many more cases to consider then in the previous cases so for the sake of brevity (and to save the reader) for each case we will  give the necessary changes that need to be made to find a sequencing, but will not verify that the new orderings work.  This is generally straightforward and is similar to the verifications in the previous theorem.  A  detailed verification of all the cases is given in \cite{amelia-thesis}.

\begin{thm} Conjecture \ref{conjecture1} is true when $k =6$. \end{thm}

\proof
Let $A = \{a_1, a_2, a_3, a_4, a_5, a_6\}$, and let $s_i$ be the partial sum of the first $i$ numbers in an ordering of $A$. Let $p$ be the number of pairs $\{x, -x\}$ in this ordering so $p = 0,1,2,$ or 3. First note that $s_i \not = s_{i+1}$ for any $1 \leq i \leq 5$ since $0 \not \in A$. Also note that if $A$ is ordered such that for all $i$, $a_i \not = -a_{i+1}$, then $s_i \not = s_{i+2}$ for any $1 \leq i \leq 4$. Assuming this, we must only check the cases $s_1 = s_4, s_1 = s_5, s_1 = s_6, s_2 = s_5, s_2 = s_6,$ and $s_3 = s_6$.

Assume $p=0$, and let $A = \{u,v,w,x,y,z\}$. Order $A$ as $(u,v,w,x,y,z)$, renaming if necessary, so that $s_1, s_2, s_3, $ and $s_4$ are distinct. In this case since there are no occurrences of a pair $\{x, -x\}$, the only conditions that can fail are the following six possibilities: (1) $s_1 = s_5$ and $s_3 = s_6$, (2) $s_1 = s_5$ and $s_3 \not = s_6$, (3) $s_1 \not = s_5$ and $s_3 = s_6$, (4) $s_1 = s_6$, (5) $s_2 = s_5$, or (6) $s_2 = s_6$. It is straightforward to show that in each of these cases the other possibilities are mutually exclusive. We will look at each case individually. For all cases, let $s_i'$ and $s_i''$ denote the $i^{th}$ partial sum after one ($'$) or two ($''$) changes of ordering, denoted $A'$ and $A''$ respectively.
\begin{enumerate}
 
 \item $(s_1 = s_5$ and $s_3 = s_6)$: In this case we have $v+w+x+y =0 = x+z+y$. Now consider the ordering $A' = (u,v,x,w,z,y)$. Here both $s_3$ and $s_5$ have changed. Clearly, $s_1' \not = s_5'$ as $s_1' = s_1 = s_5 \not = s_5'$. Also, $s_2' \not = s_5'$ since   $s_2'  = s_5'$  would imply $x+w+z = 0$; however, since $x+z+y = 0$ this means $w = y$, a contradiction. Finally, $s_3' \not = s_6'$ since $s_3' \not = s_3 = s_6 = s_6'$.

\item $(s_1 = s_5$ and $s_3 \not = s_6)$: In this case we have that $v+w+x+y =0$. Now consider the ordering $A' = (u,v,w,x,z,y)$. First note that only $s_5$ has changed, and so we only need to check conditions containing $s_5'$. Clearly, $s_1' \not = s_5'$ since $s_1' = s_1 = s_5 \not = s_5'$. We could however have $s_2' = s_5'$. If this is the case, then $v+w+x+y = 0 = w+x+z$. Then order $A$ as $A'' = (u,w,v,x,z,y)$. Here only $s_2'$ has changed from the previous ordering, so we need only check conditions containing $s_2''$. We see that $s_2'' \not = s_5''$ as $s_2'' \not = s_2' = s_5' = s_5''$. We also see that $s_2'' \not = s_6''$ since if not, then we get $v+x+z+y = 0$; however, since $v+w+x+y = 0$ we would have $w = z$, a contradiction.

 \item $(s_1 \not = s_5$ and $s_3 = s_6)$: In this case $x+y+z=0$. Now arrange $A$ as $A' = (u,v,x,w,y,z)$. Here only $s_3$ has changed, but $s_3' \not = s_6'$ as $s_3' \not = s_3 = s_6 = s_6'$.

\item $(s_1 = s_6)$: Here we have that $v+w+x+y+z = 0$. Now order $A$ as $A' = (v,u,w,x,y,z)$. Note that only $s_1$ has changed, so we only need to check the conditions containing $s_1'$, including $s_1' = s_4'$. Clearly, $s_1' \not = s_6'$ since $s_1' \not = s_1 = s_6 = s_6'$. However, it is possible for $s_1' = s_4'$ or $s_1' = s_5',$ but note that these cases are mutually exclusive.
\begin{enumerate}
\item  $(s_1 = s_6$ and  $s_1' = s_4')$: In this case we get $v+w+x+y+z = 0 = u + w+x$. Order $A$ as $A'' = (v,u,w,y,x,z)$. Note that only $s_4'$ has changed from $A'$. Thus we only check $s_1'' = s_4''$. But this is impossible since $s_1'' = s_1' = s_4' \not = s_4''$.

\item $(s_1 = s_6$ and $s_1' = s_5')$: In this case we get $v+w+x+y+z = 0 =u+w+x+y$. Order $A$ as $A'' = (v,u,w,y,z,x)$. Here only $s_4'$ and $s_5'$ have changed from the previous arrangement. We see that $s_1'' \not = s_4''$, since equality implies that $u+w+y = 0$ and hence $x = 0$, a contradiction. Also, $s_1'' \not = s_5''$ since if not, then we have that $u+w+y+z = 0$; however, since $u+w+x+y = 0,$ this implies $z = x$, which is impossible. Finally, $s_2'' \not = s_5''$ as equality would imply that $w+y+z = 0$, but since $v+w+x+y+z = 0$ we would have $v = -x$, which is a contradiction.
\end{enumerate}

\item $(s_2 = s_5)$: In this case we have $w+x+y = 0$. Now order $A$ as $A' = (u,v,w,x,z,y)$. Note that only $s_5$ has changed so we only need to check those cases involving $s_5'$. Clearly $s_2' \not = s_5'$ as $s_2' = s_2 = s_5 \not = s_5'$. However, it is possible for $s_1' = s_5'$. In this case we get $w+x+y =0 = v+w+x+z$. Reorder $A$ as $A'' = (u,v,w,z,y,x)$. Here only $s_4'$ and $s_5'$ have changed from $A'$. We see $s_1'' \not = s_4''$ as equality would imply  that $v+w+z = 0$ and since $v+w+x+z = 0$ this would imply $x=0$, which is impossible. Also, $s_1'' \not = s_5''$ as $s_1'' = s_1' = s_5' \not = s_5''$. Finally, $s_2'' \not = s_5''$ as $s_2'' = s_2 = s_5 \not = s_5''$.

\item $(s_2 = s_6)$: In this case $w+x+y+z = 0$. Order $A$ as $A' = (u,w,v,x,y,z)$. Here only $s_2$ has changed and thus we need only check the cases containing $s_2'$. We see that $s_2' \not = s_6'$ as $s_2' \not = s_2 = s_6 = s_6'$. It is possible for $s_2' = s_5'$. In this case we have $w+x+y+z=0=v+x+y$. Reorder $A$ as $A'' = (u,w,v,x,z,y)$. Here only $s_5'$ has changed from $A'$. We see $s_1'' \not = s_5''$ as equality would imply $w+v+x+z = 0$; however, since $w+x+y+z=0$ this would mean $v = y$, a contradiction. Clearly, $s_2'' \not  = s_5''$ as $s_2'' = s_2' = s_5' \not = s_5'$. This completes the case for $p = 0$.
\end{enumerate}

Next assume that $p = 1$. Let $A = \{x,-x,v,w,y,z\}$ and order $A$ as $(x,v,-x,w,y,z)$. Since $x$ is not adjacent to $-x$, the only conditions that can fail are the following nine possibilities: (1) $s_1 = s_4$ and $s_2 = s_6$ (2) $s_1 = s_4$ and $s_3 = s_6$, (3) $s_1 = s_4$, $s_2 \not = s_6$, and $s_3 \not = s_6$, (4) $s_1 \not= s_4$ and $s_2 = s_6$, (5) $s_1 \not = s_4$, $s_1 \not = s_5$, and $s_3 = s_6$, (6) $s_1 = s_5$ and $s_3 = s_6$, (7) $s_1 = s_5$ and $s_3 \not = s_6$, (8) $s_1 = s_6$, or (9) $s_2 = s_5$. It is straightforward to show that no other combinations are possible. We consider each case individually and define $s_i'$ and $s_i''$ as before.
\begin{enumerate}
\item ($s_1 = s_4$ and $s_2 = s_6$): In this case we have $x = v+w = w+y+z$. Order $A$ as $A' = (x, w, y, v, -x, z)$. Here $s_2, s_3$, and $s_4$ have changed. Clearly $s_1' \not = s_4'$ as $s_1' = s_1 = s_4 \not = s_4'$ and $s_2' \not = s_6'$ as $s_2' \not = s_2 = s_6 = s_6'$. Also, $s_2' \not = s_5'$ since equality would imply that $x = y+v$ and since $x = v+w$ this implies that $w =y$, which is impossible. Finally, $s_3' \not = s_6'$ since if $s_3' = s_6'$, then $x = v+z$, and since $x = v+w$ this would mean $w=z$, a contradiction.

\item ($s_1 = s_4$ and $s_3 = s_6$): In this case $x = v+w$ and $w+y+z = 0$. Then order $A$ as follows: $A' = (x,v,w,y,-x,z)$. Here only $s_3$ and $s_4$ have changed. Clearly $s_1' \not = s_4'$ since $s_1' = s_1 = s_4 \not = s_4'$ and similarly $s_3' \not = s_6'$ as $s_3' \not = s_3 = s_6 = s_6'$.

\item ($s_1 = s_4$ and $s_2 \not = s_6$ and $s_3 \not = s_6$): In this case we get $x = v+w$. Order $A$ as $A' = (x,v,w,y,-x,z)$. Here only $s_3$ and $s_4$ have changed, so we need only check cases containing $s_3'$ and $s_4'$. Clearly, $s_1' \not = s_4'$ as $s_1' = s_1 = s_4 \not = s_4'$. However, it is possible for $s_3' = s_6'$. In this case we have $x = v+w = y+z$. Now order $A$ as $A'' = (x,w,y,v,-x,z)$. Here $s_2'$ and $s_3'$ have changed from the previous ordering. We see $s_2'' \not = s_5''$ as equality  would imply $x = y+v$; however, since $x = v+w$ this would mean $w = y$, which is impossible. Also, if $s_2''  = s_6''$, then we would have $x = y+v+z$ and since $x = y+z$ this would imply $v = 0$, a contradiction. Hence $s_2'' \not = s_6''$ . Finally, $s_3'' \not = s_6''$ since $s_3'' \not = s_3' = s_6' = s_6''$.

\item ($s_1 \not = s_4$ and $s_2 = s_6$): In this case we have that $x = w+y+z$. Order $A$ as  $A' = (x,w,v,-x,y,z)$. Here only $s_2$ and $s_3$ have changed. Clearly, $s_2' \not = s_6'$ since $s_2' \not = s_2 = s_6 = s_6'$. Also, $s_3' \not = s_6'$ as equality would imply $x = z+y$; however, since $x = w+y+z$, this would mean $w=0$, which is impossible. It is possible however for $s_2' = s_5'$. In this case we get $x = w+y+z=v+y$. Reorder $A$ as $A'' = (x,w,v,-x,z,y)$. Here only $s_5'$ has changed. We see $s_1'' \not = s_5''$ since if $s_1''  = s_5''$, then $x = w+y+z$ and since $x = w+y+z$ we get that $y = v$, a contradiction. Finally, $s_2'' \not = s_5''$ since $s_2'' = s_2' = s_5' \not = s_5''$.

\item ($s_1 \not = s_4$, $s_1 \not = s_5$, and $s_3 = s_6$): In this case we get $w+y+z =0$. Order $A$ as $A' = (x,v,w,-x,y,z)$. Here only $s_3$ has changed. Clearly, $s_3' \not = s_6'$ since $s_3' \not = s_3 = s_6 = s_6'$.

\item($s_1 = s_5$ and $s_3 = s_6$): In this case we get $x = v+w+y$ and $w+y+z = 0$. We order $A$ as $A' = (x,v,w,-x,z,y)$.  Here only $s_3$ and $s_5$ have changed. We see $s_1' \not = s_5'$ since $s_1' = s_1 = s_5 \not = s_5'$. Similarly, $s_3' \not = s_6'$ as $s_3' \not = s_3 = s_6 = s_6'$. Also, $s_2' \not = s_5'$ as equality would mean that $x = w+z$. But since $w+y+z = 0$, we have that $w+z = -y$. Together these imply that $x = -y$, a contradiction.

\item ($s_1 = s_5$ and $s_3 \not = s_6$): In this case $x = v+w+y$. Order $A$ as $A' = (x,v,-x,w,z,y)$. Here only $s_5$ has changed. Clearly, $s_1' \not = s_5'$ since $s_1' = s_1 = s_5 \not = s_5'$. It is possible however for $s_2' = s_5'$. In this case we get $x = v+w+y=w+z$. Now order $A$ as $A'' = (x,v,z,-x,y,w)$. Here $s_3', s_4',$ and $s_5'$ have changed from $A'$. We see $s_3'' \not = s_6''$ as equality would imply $x = y+w$, but since $x = w+z$ this would mean $z=y$, which is a contradiction. Also, $s_1'' \not = s_4''$ since if $s_1''  = s_4''$, then $x = v+z$; however,  since $x = w+z$ this would mean $w = v$, which is impossible. Furthermore, $s_1'' \not = s_5''$ as equality would imply $x = v+z+y$ and  since $x = v+w+y$ this would imply that $w=z$, a contradiction. Finally, $s_2'' \not = s_5''$ since $s_2'' = s_2' = s_5' \not = s_5''$.

\item ($s_1 = s_6$): In this case we get $x = v+w+y+z$. Then order $A$ as $A' = (v,x,w,-x,y,z)$. Here only $s_1$ and $s_3$ have changed. We see $s_1' \not =  s_4'$ since this would imply $w =0$, which is impossible. Also, $s_1' \not = s_5'$ as this means $w = -y$, a contradiction. Clearly, $s_1' \not = s_6'$ since $s_1' \not = s_1 = s_6 = s_6'$. Finally, $s_3' \not = s_6'$ as equality would imply $x = y+z$ and since $x =v+w+y+z$ this would mean $v = -w$, a contradiction.

\item ($s_2 = s_5$): In this case $x = w+y$. Order $A$ as  $A' = (x,v,-x,w,z,y)$. Here only $s_5$ has changed. Clearly, $s_2' \not = s_5'$ as $s_2' = s_2 = s_5 \not = s_5'$. However, it is possible for $s_1' = s_5'$. In this case we get $x = w+y = v+w+z$. Now arrange $\A$ as $A'' = (v,x,w,-x,z,y)$. Here only $s_1'$ and $s_3'$ have changed. We see $s_1'' \not = s_4''$ as equality would imply $w =0$, a contradiction. Clearly, $s_1'' \not = s_5''$ since $s_1'' \not = s_1' = s_5' = s_5''$. Also, if $s_1'' = s_6''$, then  $w+z+y = 0$, but since $x = w+y$ this would mean $-x = z$, which is impossible. Hence $s_1'' \not = s_6''$.  Finally, $s_3'' \not = s_6''$ as equality would mean that $x = z+y$; however, since $x = w+y$ this would imply $w = z$,  a contradiction. This completes the case $p=1$.
\end{enumerate}

Now assume $p = 2$. Let $A = \{x,-x,y,-y,w,z\}$ and order $A$ as $A' = (x,y,-x,-y,w,z)$. Since neither $x, -x$ nor $y, -y$ are adjacent in $\A$, we need only check those partial sums at least three apart. Clearly, $s_1 \not = s_4$ since that implies $x = 0$, and $s_1 \not = s_5$ since that yields $x = w$. The only conditions which could fail are the following four possibilities: (1) $s_1 = s_6$, (2) $s_2 = s_5$, (3) $s_2 = s_6$, and (4) $s_3 = s_6$. It is straightforward to show that if any one of these conditions hold, then the other three do not hold. We look at each individual case.

\begin{enumerate}
\item ($s_1 = s_6$): In this case we have that $x = w+z$. We order $A$ as $A' = (w,x,y,-x,z,-y)$. Here every partial sum except $s_6$ has changed. Clearly, $s_1' \not = s_4'$ since this would mean $y = 0$, $s_1' \not = s_5'$ since this would mean $y = -z$, $s_1' \not = s_6'$ since this would mean $z = 0$, and $s_2' \not = s_6'$ since this would mean $x = z$. We also see $s_2' \not = s_5'$ as equality  would imply $x = y+z$ and since $x = w+z$, this would mean $y = w$. Finally, we see $s_3' \not = s_6'$ since if $s_3' = s_6'$, then $z = x+y$ and since $x = w+z$ this would imply $y=-w$, a contradiction.

\item ($s_2 = s_5$): Here $w = x+y$. Then order $A$ as $A' = (x,y,-x,-y,z,w)$. Here only $s_5$ has changed. We see $s_1' \not = s_5'$ as equality would mean that $x=z$, a contradiction. Also, $s_2' \not = s_5'$ since $s_2' = s_2 = s_5 \not = s_5'$.

\item ($s_2 = s_6$): Here $x = w+z -y$. Order $A$ as $A' = (x,z,-y,w,y,-x)$. Here $s_2, s_3, s_4$, and $s_5$ have all changed. We see $s_2' \not = s_5'$ as this would imply $w = 0$ and $s_2' \not = s_6'$ as this would imply $x =w$. Also, $s_1' \not =  s_4'$ as equality would imply $z-y+w = 0$; however, since $x = w+z-y$ this would mean $x=0$. Furthermore, $s_1' \not = s_5'$ as equality would imply $z = -w$, a contradiction. It is however possible for $s_3' = s_6'$. In this case we get $x = w+z-y$ and $x = w+y$, which implies $z = 2y$. Now order $A$ as $A'' = (x,w,y,z,-x,-y)$. Again, $s_2', s_3', s_4'$, and $s_5'$ have all changed from the previous ordering. We see $s_2'' \not = s_5''$ as equality would imply $x = y + z$ and since $x = w+z-y$ this means $w = 2y$.  But since $z = 2y$ this implies $w = z$. Also, $s_2'' \not = s_6''$ since here equality would imply $x = z$ and $s_3'' \not = s_6''$ since $s_3'' \not = s_3' = s_6' = s_6''$. Furthermore, $s_1'' \not =  s_4''$ since if $s_1''  =  s_4''$, then  this would imply that $w+y+z=0$. But since $x = w+z-y$ we get that $x = -2y$; however, since $z = 2y$ this would mean $x = -z$. Finally, $s_1'' \not = s_5''$ since equality would imply $x = w+y+z$; however, since $x = w+z-y$ this means $y = -y$, a contradiction.

\item ($s_3 = s_6$): In this case we have $y = w+z$. Order $A$ as $A' = (w,x,y,-x,z,-y)$. Here everything but $s_6$ has changed. Clearly, $s_1' \not = s_4'$ as this would imply $y=0$, $s_1' \not = s_5'$ since this would mean $y = -z$, and $s_1' \not = s_6'$ as this implies $z=0$. Also, $s_2' \not = s_6'$ as equality would imply $x=z$. Furthermore, $s_3' \not = s_6'$ as equality means $y = z-x$ and since $y = w+z$ this would imply $w= -x$. It is possible however for $s_2' = s_5'$. In this case we get $y = w+z$ and $x = z+y$. Reorder $A$ as $A'' = (x,y,w,-x,z,-y)$. Here $s_1'$ and $s_2'$ have changed from the ordering $A'$. We see $s_1'' \not = s_4''$ as equality would imply $x = y+w$ and since $x = z+y$ this means $z = w$. Also, $s_1'' \not = s_5''$ since equality implies that $x = y+w+z$; however, since $x = z+y$ this implies $w = 0$, a contradiction. Furthermore, $s_1'' \not = s_6''$ since equality would imply $x = w+z$ and since $y = w+z$ we have that $x = y$, which is impossible. Clearly, $s_2'' \not = s_5''$ since $s_2'' \not = s_2' = s_5' = s_5''$. Finally, $s_2'' \not = s_6''$ since equality implies $x+y = w+z$; however, since $y= w+z$, this would mean $x=0$, which is a contradiction. This completes the case for $p =2$.
\end{enumerate}

Finally, assume that $p = 3$. Let $A = \{x,-x, y, -y, z, -z\}$ and order $A$ as $A' = (x,y,z,-x,-y,-z)$. Since no pair of additive inverses appears in adjacent positions, we only need to check the partial sums that are least three apart. Clearly, $s_1 \not = s_5$ since this would imply $x = z$, $s_1 \not = s_6$ as this would imply $x = 0$, and $s_2 \not = s_6$ since this would imply $x = -y$. The only conditions that can fail are the following three possibilities: (1) $s_1 = s_4$, (2) $s_2 = s_5$, or (3) $s_3 = s_6$. It is straightforward to show these possibilities are mutually exclusive. We consider each case individually. 
\begin{enumerate}

\item $(s_1 = s_4$): In this case we get $x = y+z$. Then reorder $A$ as  $A'' = (x,y,z,-y,-x,-z)$. Here only $s_4$ has changed and clearly $s_1' \not = s_4'$ since $s_1' = s_1 = s_4 \not = s_4'$. 

\item ($s_2 = s_5)$: In this case we have that $z = x+y$. Now order $A$ as $A'' = (x,-y,z,y,-x,-z)$. Here $s_2, s_3,$ and $s_4$ have changed. We see $s_1' \not = s_4'$ since this would imply $z=0$. Also, $s_2' \not = s_5'$ as $s_2' \not = s_2 = s_5 = s_5'$ Furthermore, $s_2' \not = s_6'$ as this would imply $x=y$. Finally, $s_3' \not = s_6'$ as equality implies that $x-y+z=0$.  But since $z = x+y$, then $x = z-y$, which implies that $z = y$, a contradiction.

\item ($s_3 = s_6$): Here $x+y+z =0$. Order $\A$ as $A' = (x,y,-z,-x,z,-y)$. Here $s_3, s_4,$ and $s_5$ have changed. Clearly, $s_3' \not = s_6'$ as $s_3' \not = s_3 = s_6 = s_6'$. Also, $s_1' \not = s_4'$ as equality would imply $x = y - z$; however,  since $x+y+z = 0$ this means $y = -y$. Furthermore, $s_1' \not = s_5'$ as equality would imply that $x=y$ and $s_2' \not = s_5'$ as equality here would imply $x =0$. This completes the proof. \qed
\end{enumerate}

It does not appear promising to try to extend the proof to the case $k = 7$.

\section{Random Subsets}
\label{sec4}

The next theorem is probabilistic in nature and shows that a randomly chosen subset $A$ of size $k$ is orderable if $k$ is not too large.  First, we state and prove a useful lemma.

\begin{lemma}\label{lemma1} Let $1\leq \ell \leq n-2$ and let $t \in \ZZ_n$.  For any set $A \in  \ZZ_n$, let $s_A$ be the sum of the elements of $A$ in $\ZZ_n$.  Then for a randomly chosen $\ell$-subset of $\ZZ_n \setminus \{0\}$, the probability that $s_A=t$ is at most $2/n$.
\end{lemma}

\begin{proof}  When $\ell=1$, this is obvious.  Next assume that $2\leq \ell \leq n/2$.  Let $B\subseteq \ZZ_n \setminus \{0\}$ with $|B| = \ell-1$.  %Let $S = s_B$. 
There are $n-\ell$ elements of $\ZZ_n \setminus \{0\}$ that extend $B$ to a subset $A$ of size $\ell$. At most one of these extensions will have $s_A = t$.  The probability that a random extension has sum equal to $t$ is thus at most $$\frac{1}{ n-\ell} \leq  \frac{1}{ n-n/2}= \frac{2}{ n},$$ as desired.   

Now assume that $n/2 < \ell \leq n-2$.   Let $r = \sum_{a\in \ZZ_n}a $ (so $r= 0$ if $n$ is odd and $r=n/2$ if $n$ is even). Again, let $A$ be a set of size $\ell$, and let $B = \ZZ_n \setminus (A \cup \{0\})$. Then $s_A = t$ if and only if $s_B = r - t$.  Since $1\leq n-1-\ell \leq n/2$, the probability of this occurring is at most $2/n$, from the previous case. 
\end{proof}

\begin{thm}  Let $A$ be a randomly chosen $k$-subset of $\ZZ_n \setminus \{0\}$.   Then
the probability that $A$ cannot be ordered in such a way that all runs are nonzero
is at most $k(k-1)/n$.
%$\binom{k}{2} \times \frac {2}{n}$. 
\end{thm}

\begin{proof}
Consider the set $\mathcal{S}$ of all orderings of all $k$-subsets of $\ZZ_n \setminus \{0\}$.
There are $k! \binom{n-1}{k}$ orderings in this set. 
%An ordering $(a_1, \dots , a_k)$ has distinct partial sums
%if and only if all runs $r_{i,j}$ have non-zero sums modulo $n$, where $1 \leq i < j \leq k$. 
Define an ordering to be {\it bad} if at least one run is zero modulo $n$.

The probability that a random run has a sum that is
 zero modulo $n$ is at most $2/n$ by Lemma \ref{lemma1}. There are 
$\binom{k}{2}$  runs to consider for each ordering. 
Therefore, if $\overline{p}$ denotes the  probability that a random ordering in $\mathcal{S}$ is bad,
then we have 
\begin{equation}
\label{EQ1}
\overline{p} \leq \binom{k}{2} \times \frac {2}{n}.
\end{equation}
The probability $\overline{p}$ is computed over all the orderings in $\mathcal{S}$. 
Now consider $\mathcal{S}$ to be partitioned
 into $\binom{n-1}{k}$ sets $\mathcal{S}_A$, each of size $k!$, where each set $\mathcal{S}_A$ 
 consists of the $k!$ orderings of 
 a fixed $k$-subset  $A \subseteq \ZZ_n \setminus \{0\}$.
Let $p_A$ denote the probability that a randomly chosen 
ordering of the $k$-subset $A$ is bad. 
It is clear that 
\begin{equation}
\label{EQ2} \overline{p} = \frac{1}{\binom{n-1}{k}} \sum_{A} p_A.
\end{equation}
Define a $k$-subset $A$  to be {\it bad} if every ordering in $\mathcal{S}_A$ 
is bad. Let $\mathcal{A}$ denote the set of bad $k$-subsets.
It is obvious that
\begin{equation}
\label{EQ3} \sum_{A} p_A \geq |\mathcal{A}|
\end{equation}
since $p_A = 1$ whenever $A \in \mathcal{A}$.
Combining (\ref{EQ1}), (\ref{EQ2}) and (\ref{EQ3}), it follows that
\[ \frac{|\mathcal{A}|}{\binom{n-1}{k}} \leq \binom{k}{2} \times \frac {2}{n}.\]
However, the probability that a random $k$-subset is bad is easily seen to be
$|\mathcal{A}|/\binom{n-1}{k}$, so we are done.
\end{proof}

As an example, if we take $k \approx \sqrt{n/2}$, then the probability that a randomly chosen $k$-subset of $\ZZ_n \setminus \{0\}$ can be ordered so that all the runs are nonzero is at least $1/2$.

\section{Ordering Subsets of $A$}
\label{sec5}

A further question concerns choosing a subset $B$ of a given set $A$ such that $B$ can be ordered
in such a way that all of its partial sums are distinct.

\begin{problem}\label{conjecture7} Given $A \subseteq \ZZ_n$, find a subset  $B\subseteq A$
of maximum size that can be ordered so all of its partial sums are distinct.  
\end{problem}

If it always holds that $B = A$, then  Conjecture \ref{conjecture1} is valid.  
We show the weaker result 
that there always exists $B \subseteq A$ satisfying the desired properties,
where $ |B| \geq (k+1)/2$, via a greedy algorithm.

\begin{thm} 
\label{con7proof}
Problem  \ref{conjecture7} always has a solution $B$ where $ |B| \geq (k+1)/2$.
\end{thm}  

\begin{proof}  Assume that the sequence $(a_1, a_2,\ldots, a_r)$ has the property that for  $1\leq i <j\leq r$, it holds that $s_i \neq s_j$. Now there are $r$ partial sums,  so if there are at least $r+1$ elements from $A$ not already used in the sequence, it is possible to choose one, say $x\in A$ such that
$s_r +x \neq s_i$ for all $i\leq r$.  This is possible if $k \geq 2r+1$ or if $r \leq (k-1)/2$.
In this case, the sequence can be extended to a sequence of length $r+1$ having distinct partial sums.
\end{proof}

\begin{thm}
For any $A\subseteq \zed_n \setminus \{0\}$ with $|A| = 2t$, there exist 
at least $2^t$ $t$-subsets $B \subseteq A$ that can be ordered so their partial sums are distinct.
\end{thm}

\begin{proof} Similar to the proof of Theorem \ref{con7proof}, given a sequence of length $r$ having distinct partial sums, there are at least $2t-2r$ ways to extend it to a sequence of length $r+1$ having distinct partial sums. 
We get at least \[2t \times (2t-2) \times \cdots \times 2= 2^t \, t!\] permissible orderings 
of $t$-subsets $B \subseteq A$.
Clearly, any given $t$-subset $B$ occurs at most $t!$ times in this list.
Therefore, there are at least $2^t$ different $t$-subsets $B \subseteq A$ that can be ordered so that
the partial sums are distinct. 
\end{proof}

A similar (but slightly messier) result can be proven when $|A|$ is odd.

\section{Sums of Elements in $k$-subsets}
\label{sec6}

In this section, we consider a different but related problem.  In the next theorem, we will show that if $n$ is a prime, then among all the $k$-subsets of $\ZZ_n \setminus \{0\}$, the  sums are almost equally distributed. In particular, we can compute the number of subsets of $\ZZ_n \setminus \{0\}$ whose sum is $0$. These results can be viewed as
more precise versions of Lemma \ref{lemma1} for the cases where $n$ is prime.

Let $p$ be a prime with  $F_p$ the finite field of order $p$ and $F_p^*$ the multiplicative group of the field.
Let $S_k(\alpha)$ denote the set of all $k$-subsets of $F_p$ whose sum is $\alpha$ and let $N_k(\alpha) = |S_k(\alpha)|.$
Similarly let $S_k^*(\alpha)$ denote the set of all $k$-subsets of $F_p^*$ whose sum is $\alpha$ and let $N_k^*(\alpha) = |S_k^*(\alpha)|.$

\begin{lemma}\label{lemmax}   $N_k^*(\alpha) = N_k^*(\beta)$ for any $\alpha, \beta\in F_p^*$. \end{lemma}

\begin{proof}  Let $\alpha \in F_p^*$ and let $S \in S_k^*(1)$.  Then, if $\alpha S= \{\alpha s : s \in S\}$, we see that $\alpha S \in S_k^*(\alpha)$.  So $S \rightarrow \alpha S$ is clearly a bijection from $S_k^*(1)$ to $S_k^*(\alpha)$.  Hence, for every $\alpha \in F_p^*$, we have that $N_k^*(\alpha) = N_k^*(1)$.  The result follows.
\end{proof}

So all the $N_k^*(\alpha)$s are equal when $\alpha \neq 0$.  We next show the same conclusion holds for subsets of $F_p$ (which may now include $0$).  The proof is the same, except the bijection is additive instead of multiplicative.  Here we denote $a + S= \{a + s : s \in S\}$ for $a \in F_p$.

\begin{lemma}\label{lemmay}   $N_k(\alpha) = N_k(\beta)$ for any $\alpha, \beta\in F_p$. \end{lemma}

\begin{proof} Let $\alpha \in F_p$ and let $S \in S_k(0)$.  Since $p$ is prime there exists a $\beta \in F_p$ such that $k\beta = \alpha$.  Hence $\beta  + S \in S_k(k\beta )= S_k(\alpha )$.  So $S \rightarrow \beta+ S$ is clearly a bijection from $S_k(0)$ to $S_k(\alpha)$.  Thus for every $\alpha \in F_p$ we have that $N_k(\alpha) = N_k(0)$.  The result follows.
\end{proof}

%We can also prove this the same way as previous lemma.

%\bigskip

From Lemma \ref{lemmay}, we have that $N_k(\alpha) = \frac{1}{p} \binom{p}{k}$ for every $\alpha\in F_p$.   We are now ready to prove our main result about the value of $N_k^*(0)$.

\begin{lemma}\label{lemmaz}  Let $\alpha \in  F_p^*$.   Then $$N_k^*(0) = \left\{  \begin{array}{ll} 
                                                           N_k^*(\alpha)+1 & \mbox{ if $k$ is even} \\
							  N_k^*(\alpha)-1 & \mbox{ if $k$ is odd.}   \end{array} \right. $$
\end{lemma}

\begin{proof} We prove this by induction on $k$.  When $k=1$, $N_1^*(0) = 0$, while $N_1^*(\alpha) = 1$, as desired. When $k=2$,  $N_2^*(0) = \frac{p-1}{2}$ since all elements in $F_p^*$ can be paired with their additive inverse to add to 0. But $N_2^*(2) = \frac{p-3}{2}$ since none of $0, 1,$ or 2 can be in a pair that adds to 2.  Hence $N_2^*(0) =  N_2^*(\alpha)+1$ for all $\alpha \in F_p^*$ by Lemma \ref{lemmax} above.

The key observation is that for any $\alpha$, $N_k^*(\alpha)$ counts the number of $k$-subsets summing to $\alpha$ that contain 0 as well as those that do not contain 0.  But exactly $N_{k-1}^*(\alpha)$ contain 0 and exactly $N_{k}^*(\alpha)$ don't contain 0.  So for any $\alpha$, we have $N_k(\alpha) = N_{k-1}^*(\alpha) + N_k^*(\alpha)$.

Now assume that $k$ is even and let $\alpha \in  F_p^*$.  From Lemma \ref{lemmay} we have 
$
N_k(0) = N_k(\alpha) 
$
so 

$$
N_{k-1}^*(0) + N_k^*(0)= N_{k-1}^*(\alpha) + N_k^*(\alpha)
$$
thus
$$
N_k^*(0) = (N_{k-1}^*(\alpha) - N_{k-1}^*(0)) + N_k^*(\alpha)
$$ 
and hence when $k $ is even we have by induction that 

$$
N_k^*(0) =  1+ N_k^*(\alpha)
$$
since $k-1$ is odd.  The case when $k$ is odd is similar.  \end{proof}

The following theorem is now immediate from Lemmas \ref{lemmax} and \ref{lemmaz}.

\begin{thm}  When $k$ is even, $N_k^*(0) = \frac{1}{p}( \binom{p-1}{k}+1)$. When $k$ is odd, 
$N_k^*(0) = \frac{1}{p}( \binom{p-1}{k}-1)$. \end{thm}

\bigskip
\noindent{\bf Acknowledgement:}  The authors would like to thank Brian Alspach, Ian Wanless, Daniel Horsley and Diane Donovan for useful discussions on this topic.

\end{document}